\documentclass[smallextended,numbook,runningheads]{svjour3}     
\smartqed  
\usepackage{epsfig,amsfonts,amssymb,latexsym,amsmath}
\usepackage{graphicx}
\usepackage{epstopdf} 
\usepackage{algorithm}
\usepackage{algorithmic}
\usepackage{subfigure}
\usepackage{mathptmx}      

\journalname{BIT}
\begin{document}

\title{Scaled Fixed Point Algorithm for Computing the Matrix Square Root
}


\author{Harry Oviedo        \and
        Hugo Lara           \and
        Oscar Dalmau
}


\institute{H. Oviedo \at
              Mathematics Research Center, CIMAT A.C. Guanajuato, Mexico \\
              \email{harry.oviedo@cimat.mx}
           \and
           H. Lara \at
              Universidade Federal de Santa Catarina. Campus Blumenau. Brasil. \\
              \email{hugo.lara.urdaneta@ufsc.br}
           \and
           O. Dalmau \at
              Mathematics Research Center, CIMAT A.C. Guanajuato, Mexico \\
              \email{dalmau@cimat.mx }
}

\date{Received: date / Accepted: date}
\maketitle

\begin{abstract}
This paper addresses the numerical solution of the matrix square root problem. Two fixed point iterations are proposed by rearranging the nonlinear matrix equation $A - X^2 = 0$ and incorporating a positive scaling parameter. The proposals only need to compute one matrix inverse  and at most two matrix multiplications per iteration. A global convergence result is established. The numerical comparisons versus some existing methods from the literature, on several test problems, demonstrate the efficiency and effectiveness of our proposals.
\end{abstract}

\section{Introduction}
\label{sec:1}
In this paper, we derive novel fixed point algorithms for numerically approximate the solution of the matrix square root problem. Given a square matrix $A\in\mathbb{R}^{n\times n}$, we address the problem of finding a matrix $X\in\mathbb{C}^{n\times n}$ such that it satisfies the following quadratic system of equations
\begin{equation}
  X^{2} = A. \label{problem}
\end{equation}
Our approach lies on the special case of \eqref{problem} when $A$ is a symmetric positive semi--definite (PSD) matrix with real entries. The matrix square root problem plays an important role in many applications, and arises in several contexts such as: computation of the matrix sign function \cite{Higham}, signal processing applications \cite{PanChen,VanDerMerwe}, parallel translation and polar retractions for optimization on Riemannian manifolds \cite{iannazzo2017riemannian,yuan2016riemannian,zhu2017riemannian}, the Karcher mean computation \cite{iannazzo2017riemannian}, among others.\\

It is well--known that the system \eqref{problem} has no a unique solution (if one exists). However, if $A$ is positive semi--definite then problem \eqref{problem} has a unique solution, denoted by $A^{1/2}$. This seems to be the most frequent case in practice.\\

The most numerically stable way to solve problem \eqref{problem} is via the Schur decomposition. This strategy reduces the problem \eqref{problem} into the computation of the matrix square root of an upper triangular matrix. Specifically, let $A = UTU^{*}$ a Schur decomposition of $A$, where $T$ is upper triangular and $U$ is unitary. Then, observe that $A^{1/2} = UT^{1/2}U^{*}$. A blocked Schur procedure for solving \eqref{problem} is presented in \cite{deadman2012blocked}. When $A\in\mathbb{R}^{n\times n}$ is symmetric, this method is reduced to compute an eigenvalue decomposition of $A$. However, this strategy is impractical for $n$ large. Therefore, there is not other option than to resort to iterative methods.\\

Several types of iterative algorithms have been introduced to address the problem \eqref{problem}. Possibly the first one is the Newton method, developed by Higham in \cite{higham1986newton}, which constructs a sequence of iterates by the following recurrence
\begin{equation}
X_{k+1} = \frac{1}{2}(X_{k} + X_k^{-1}A), \quad \textrm{starting at } X_0 = A. \label{Newton}
\end{equation}
This method enjoys a quadratic convergence rate to $A^{1/2}$ under some assumptions, see \cite{higham1986newton}. However, Newton iteration suffers from instability near the solution, and absence of global convergence. In an attempt to overcome these drawbacks, in \cite{Higham} is introduced a scaled Newton method to approximate the solution of \eqref{problem} via polar decomposition. First it is computed the Cholesky factorization $A = L^{\top}L$, to obtain the square root as $A^{1/2} = UL$, where $U$ is the limit of the sequence $\{U_k\}$ generated by
\begin{equation}
U_{k+1} = \frac{1}{2}\left(\mu_kU_{k} + \frac{1}{\mu_k}U_k^{-\top}\right), \quad \textrm{starting at } U_0 = L, \label{ScaledNewton}
\end{equation}
where $\mu_k>0$ is the scaling parameter.\\

In \cite{Sra}, Sra developed a fixed point iteration, which is related to a non--convex optimization problem.  Starting at $X_0 = \frac{1}{2}(A+I)$, Sra's iteration solve \eqref{problem} by running
\begin{equation}
X_{k+1} =  [(X_k+A)^{-1}+(X_k+I)^{-1}]^{-1}. \label{Sra}
\end{equation}
This iteration was also considered in \cite{Ando}, in the context of geometric mean computation of positive operators, motivated by electrical resistance networks. In \cite{Sra} Sra establishes the linear convergence of \eqref{Sra} to $A^{1/2}$ based on a geometric optimization approach. Specifically, Sra cast the problem \eqref{problem} as the non--convex optimization model
\begin{equation}
\min_{X\succ 0} \mathcal{F}(X) = \delta_S^{2}(X,A) + \delta_S^{2}(X,I), \label{Sra2}
\end{equation}
whose unique solution is the desired square root $X^{*} = A^{1/2}$. Here, $\delta_S^{2}(\cdot,\cdot)$ denotes the $S$--divergence (see \cite{sra2016positive}) defined by
\begin{equation}
\delta_S^{2}(X,Y) = \log \det\left(\frac{X+Y}{2}\right) - \frac{1}{2}\log \det\left(\frac{XY}{2}\right). \label{Sra3}
\end{equation}
The first--order optimality conditions associated to \eqref{Sra2} leads to the following matrix equation,
\begin{equation}
\frac{1}{2}\left(\frac{X+A}{2}\right)^{-1} + \frac{1}{2}\left(\frac{X+I}{2}\right)^{-1} - X^{-1} = 0. \label{Sra4}
\end{equation}
Direct manipulation of this Riccati equation leads to the Sra's iteration \eqref{Sra}.\\

Another  first--order  method was proposed in \cite{jain2015computing}. Namely, the classical steepest descent method (SD) for minimizing the least--square problem
\begin{equation}
\min ||X^2 - A||_F^{2},  \quad s.t. \quad X\succeq 0, \label{GradientIteration}
\end{equation}
related to the problem \eqref{problem}. The main advantage of the steepest descent method to solve \eqref{GradientIteration} over the Newton and Sra methods, is that the SD method does not require computing a matrix inverse per iteration, which makes the SD method an attractive procedure. However, the optimal step--size depends on a certain constant $c>0$, whose existence is theoretically guaranteed (see Theorem 3.2 in \cite{jain2015computing}), leading to the absence of a closed formula for the step--size in practice. Although strategies of sufficient descent as the standard Armijo--rule with a backtracking strategy can be used, this could cause the SD algorithm performs many matrix multiplications per iteration, which is not desired to design an efficient method.\\

Recently, Gawlik in \cite{gawlik2019zolotarev} introduced the Zolotarev iterations for finding the matrix square root, this approach is based on a recursion for rational approximations of $\sqrt{a}$, which is extended to the matrix case. This procedure is similar to the Pad\`{e}'s iterations \cite{higham1997stable,laszkiewicz2009pade}, but converges more rapidly to the solution for matrices that have eigenvalues with widely varying magnitudes.\\

In this paper, we introduce two new fixed point methods to compute a numerical solution of problem \eqref{problem}. Based on the Sra' iteration \eqref{Sra} and keeping in mind the high computational cost per iteration of \eqref{Sra}, we propose some fixed point methods equipped with a scaling parameter, which only need to compute one inverse matrix and at most two matrix multiplications per iteration. In addition, we establish a global convergence result under the Thompson metric, following the idea of the Sra' demonstration in \cite{Sra}. Furthermore, we perform some numerical comparisons of our proposals with other state--of--the--art methods in order to demonstrate the effectiveness and efficiency of our procedures. Several numerical experiments show that our proposals are more efficient than the Sra iteration and also converge faster to the solution of \eqref{problem}.\\

The rest of this paper is organized as follows. In section \ref{sec:2} we introduce our two fixed point iterations for solving the square root problem \eqref{problem}. A convergence analysis is given in section \ref{sec:3}. Some numerical tests on several experimental problems are presented in section \ref{sec:4}. Finally, conclusions are drawn in section \ref{sec:5}.

\section{Two fixed point methods}
\label{sec:2}
In this section, we introduce two fixed point methods to deal with the numerical solution of the matrix square root problem \eqref{problem}. Motivated by the Sra' iteration and looking for numerical efficiency, we construct new iterative schemes from the matrix equation \eqref{problem}. Let $\mu$ be a positive parameter (conveniently chosen). Adding the term $\mu X$ on both sides of the equation \eqref{problem}, and then pre--multiplying by the inverse matrix $(X+\mu I)^{-1}$, we arrive at
\begin{equation}
  X = (A+\mu X)(X+\mu I)^{-1}, \label{eq2_sec2}
\end{equation}
which leads us to our first fixed point iteration, starting at an initial symmetric positive semi--definite matrix $X_0\in\mathbb{R}^{n\times n}$,
\begin{equation}
  X_{k+1} = (A+\mu X_k)(X_k+\mu I)^{-1}, \quad k=0,1,2,\dots \label{eq3_sec2}
\end{equation}
Following a similar reasoning, we can build another fixed point iteration to address problem \eqref{problem}. From $A = XX$, multiplying this equation by $X^{\top}$, adding $\mu X$ and then multiplying by $(X^{\top}X + \mu I)^{-1}$ we obtain,
\begin{equation}
  X = (X^{\top}X + \mu I)^{-1}(X^{\top}A + \mu X), \label{eq4_sec2}
\end{equation}
which suggests the following fixed point iterative process,
\begin{equation}
  X_{k+1} = (X_k^{\top}X_k + \mu I)^{-1}(X_k^{\top}A + \mu X_k), \quad k=0,1,2,\dots \label{eq5_sec2}
\end{equation}
From equations \eqref{eq3_sec2} and \eqref{eq5_sec2}, we note that our approaches are computationally less costly than the Sra' iteration due to our iterative procedures only require to compute a matrix inverse per iteration, while the algorithm \eqref{Sra} needs three. In addition, our proposals
incorporate a scale parameter that, if its properly selected, can speed up the convergence. Furthermore, comparing both iterative processes \eqref{eq3_sec2} and \eqref{eq5_sec2}, it is clear that each iteration of our first method is computationally more efficient than our scheme \eqref{eq5_sec2}, since it requires fewer matrix multiplications. Now we describe our efficient fixed point iterative algorithm.

\begin{algorithm}[H]
\begin{algorithmic}[1]
\REQUIRE $A\in\mathbb{R}^{n\times n}$ a given positive definite matrix, $X_0\in \mathbb{R}^{n\times n}$, $\mu>0$, $\epsilon\in(0,1)$, $k = 0$.
\ENSURE An $\epsilon$--approximate solution of the system of equations \eqref{problem}
\WHILE{ $|| A - X_k^{2} ||_F > \epsilon$ }
\STATE $X_{k+1} = (A+\mu X_k)(X_k+\mu I)^{-1}$;
\STATE $k = k +1$;
\ENDWHILE
\STATE $X^{*} = X_{k}.$
\end{algorithmic}
\caption{Fixed Point Method (FPM)}\label{alg1}
\end{algorithm}

\begin{remark}
 By changing line 2 of Algorithm \ref{alg1} for the update formula \eqref{eq5_sec2} we get our second fixed point iterative method. In this work, we will analyze this second variant only from a numerical point of view.
\end{remark}

\section{Convergence analysis}
\label{sec:3}
We now analyze Algorithm \ref{alg1} by revealing the behaviour of the residual $\delta_T(X_k,X_{*})$, where $\delta_T$ denotes the Thompson metric and $X^{*}$ is the solution of \eqref{problem}. The theoretical results provided here use similar tools than \cite{Sra}. Specifically, at the final of this section, we prove that our scheme \eqref{eq3_sec2} is a fixed--point iteration under the Thompson part metric defined by
\begin{equation}
\delta_T(X,Y)= ||\log(X^{-\frac12}YX^{-\frac12})||_2, \label{2.9}
\end{equation}
where $||\cdot||_2$ is the usual matrix norm and ``$\log$'' denotes the matrix logarithm. In the rest of this article, we will denote by $\lambda_{\textrm{min}}(M)$ and $\lambda_{\textrm{max}}(M)$ the smallest and largest eigenvalues of the symmetric matrix $M\in\mathbb{R}^{n\times n}$ respectively.\\

The following Lemma provides us some remarkable properties associated to the Thompson metric. For details about Lemma \ref{Lemma1} please see \cite{lee2008invariant,lemmens2012nonlinear,lim2012matrix,sra2015conic}.
\begin{lemma}\label{Lemma1}[Proposition 4.2 in \cite{sra2015conic}]
  Consider the Thompson metric defined in \eqref{2.9}. Let $A,B,X,Y \in \mathbb{R}^{n\times n}$ be symmetric positive definite matrices, then
  \begin{equation}\delta_T(X^{-1},Y^{-1}) =  \delta_T(X,Y), \label{9} \end{equation}
    \item \begin{equation}\delta_T(X+A,Y+B)  \leq  \max\{\delta_T(X,Y),\delta_T(A,B) \}, \label{10}\end{equation}
    and
    \item \begin{equation}\delta_T(X+A,Y+A)  \leq  \left(\frac{\alpha}{\alpha+\lambda_{\min}(A)}\right) \delta_T(X,Y), \label{11} \end{equation}
    where  $\alpha=\max\{||X||_2,||Y||_2 \}$.
\end{lemma}

Proposition \ref{Prop1} establishes another property of the Thompson metric which is fundamental to demonstrate the global convergence of our Algorithm \ref{alg1}.
\begin{proposition}\label{Prop1}
  Consider the Thompson metric defined in \eqref{2.9}. Let $A,B,X,Y \in \mathbb{R}^{n\times n}$ be symmetric positive definite matrices, then
  \begin{equation}
    \delta_T(XA^{-1},YA^{-1})=\delta_T(X,Y). \label{12}
  \end{equation}
\end{proposition}

\begin{proof}
  To see the this property, observe first that
  \begin{equation}\label{12a}
    \lambda_{\max}(AB^{-1})=\lambda_{\max}(B^{-1}A) = \lambda_{\textrm{max}}(B^{-\frac12}AB^{-\frac12}),
  \end{equation}

In fact, take first an eigenpair $(\lambda,x)$ of $AB^{-1}$, and $(\gamma,w)$
eigenpair of $B^{-1}A$. Then, for $v=B^{-\frac12}x$ we have
$$
AB^{-1}x=\lambda x \Leftrightarrow AB^{-\frac12}v=\lambda B^\frac12v \Leftrightarrow B^{-\frac12}AB^{-\frac12}v=\lambda v,
$$
which proves that $\lambda$ is an eigenvalue of $B^{-\frac12}AB^{-\frac12}$.\\

Similarly, for $y=B^{\frac12}w$, we have
$$
B^{-1}Aw=\gamma w \Leftrightarrow B^{-\frac12}AB^{-\frac12}y=\gamma y,
$$
obtaining that $\gamma$ is an eigenvalue for $B^{-\frac12}AB^{-\frac12}$.\\

Taken maximum on the Rayleigh quotient we obtain our claim.\\

On the other hand, to prove \eqref{12} observe that
$$
\lambda_{\textrm{max}}((XA^{-1})^{-1}YA^{-1})=\lambda_{\textrm{max}}(AX^{-1}YA^{-1})=\lambda_{\textrm{max}}(A^{-1}AX^{-1}Y)=\lambda_{\textrm{max}}(X^{-1}Y).
$$
In the second equality, we use the relation \eqref{12a}. Analogously we can prove
$$
\lambda_{\textrm{max}}((YA^{-1})^{-1}XA^{-1})=\lambda_{\textrm{max}}(Y^{-1}X).
$$
Then
\begin{eqnarray}
\delta_T(XA^{-1},YA^{-1}) & = & \max\{\log\lambda_{\textrm{max}}((XA^{-1})^{-1}YA^{-1}),\log\lambda_{\textrm{max}}((YA^{-1})^{-1}XA^{-1}) \} \nonumber \\
  & = & \max\{\log\lambda_{\textrm{max}}(X^{-1}Y),\log\lambda_{\textrm{max}}(Y^{-1}X) \} \nonumber \\
  & = & \delta_T(X,Y),   \nonumber
\end{eqnarray}
which completes the proof.
\end{proof}

Now consider the positive semi--definite matrix interval $\mathcal{I}=[2A(A+I)^{-1},\frac12(A+I)]$ and the
mapping $\mathcal{G}\equiv X\rightarrow (\mu X+A)(X+\mu I)^{-1}.$ We claim that
$\mathcal{G}$ maps the interval $\mathcal{I}$ to itself. In fact, if $X\in\mathcal{I}$ we have
$$
0 \prec 2A(A+\sigma_1I)(A+I)^{-1}(A+\sigma_1I)^{-1} \preceq \mathcal{G}(X) \preceq \frac12(\sigma_2A+I)(A+I)(\sigma_2A+I)^{-1},
$$
where $\sigma_1 = 1+2\mu$ and $\sigma_2 = 2+\mu$, and since $A\succ 0$ then we obtain our claim. Now we are ready to show the global convergence result for our Algorithm \ref{alg1}.

\begin{theorem}
Let $\{X_k\}_{k\geq 0}$ be the sequence generated by Algorithm \ref{alg1}, $\mu>0$ and $X^*=A^{1/2}$
be the exact solution of \eqref{problem}. Then there exist a positive constant
$\gamma\in(0,1)$ such that
$$
\delta_T(X_k,X^*)\leq \gamma^k\delta_T(X_0,X^*).
$$	
Moreover, $$\lim_{k\rightarrow\infty} X_k = X^*.$$
\end{theorem}

\begin{proof}
Consider the nonlinear map $\mathcal{G}:\mathcal{I}\rightarrow\mathcal{I}$ previously defined
and take arbitrary pair $X,Y\in \mathcal{I}$. Then using property \eqref{10} we obtain

\begin{eqnarray}
\delta_T(\mathcal{G}(X),\mathcal{G}(Y)) & = & \delta_T[(\mu X+A)(X+\mu I)^{-1},(\mu Y+A)(Y+\mu I)^{-1}] \nonumber \\
& = & \delta_T[\mu X(X+\mu I)^{-1}+A(X+\mu I)^{-1},\mu Y(Y+\mu I)^{-1}+A(Y+\mu I)^{-1}] \nonumber \\
& \leq & \max\{\delta_T(\mu X(X+\mu I)^{-1},\mu Y(Y+\mu I)^{-1}),\delta_T(A(X+\mu I)^{-1},A(Y+\mu I)^{-1} )\}. \nonumber
\end{eqnarray}

Now let us establish bounds on each  argument of the maximum: for the first one, we
use properties \eqref{9} (twice) and \eqref{11}:

\begin{eqnarray}
\delta_T(\mu X(X+\mu I)^{-1},\mu Y(Y+\mu I)^{-1}) & = & \delta_T\left(\frac1{\mu} (X+\mu I)X^{-1},\frac1{\mu} (Y+\mu I)Y^{-1}\right) \nonumber \\
& = & \delta_T\left(\frac1\mu I+X^{-1},\frac1\mu I+Y^{-1}\right) \nonumber \\
& \leq & \left( \frac{\alpha_1}{\alpha_1+\mu^{-1}} \right)\delta_T(X^{-1},Y^{-1}) \nonumber \\
& = & \left( \frac{\alpha_1}{\alpha_1+\mu^{-1}} \right)\delta_T(X,Y),  \nonumber
\end{eqnarray}
where $\alpha_1=\max\{ ||X^{-1}||_2,||Y^{-1}||_2\}.$\\

Similarly, by using properties \eqref{9}, \eqref{11} and
\eqref{12}, the second argument becomes

\begin{eqnarray}
\delta_T(A(X+\mu I)^{-1},A(Y+\mu I)^{-1}) & = & \delta_T( (X+\mu I)A^{-1},(Y+\mu I)A^{-1}) \nonumber \\
& = & \delta_T(XA^{-1}+\mu A^{-1},YA^{-1}+\mu A^{-1}) \nonumber \\
& \leq & \left( \frac{\bar\alpha_2}{\bar\alpha_2+\mu\lambda_{\textrm{min}}(A^{-1})} \right)\delta_T(XA^{-1},YA^{-1}) \nonumber \\
& = & \left( \frac{\bar\alpha_2}{\bar\alpha_2+\mu\lambda_{\textrm{max}}(A)} \right)\delta_T(X,Y), \nonumber
\end{eqnarray}
where $\bar\alpha_2=\max{||XA^{-1}||_2,||YA^{-1}||_2}$.
Let us denote $\alpha_2=\max\{||X||_2,||Y||_2\}$. Then $\bar\alpha_2\leq \alpha_2\lambda_{\textrm{min}}(A)$. Since the function $h(\alpha) = \alpha/(\alpha+c)$ is increasing, we obtain
$$\delta_T(A(X+\mu I)^{-1},A(Y+\mu I)^{-1})\leq \left(\frac{\alpha_2}{\alpha_2+\mu\kappa(A)}\right)\delta_T(X,Y),$$
where $\kappa(A)$ denotes the condition number of $A$, i.e. $\kappa(A) = \lambda_{\textrm{max}}(A)/\lambda_{\textrm{min}}(A)$.\\

Merging these two expressions in the above maximum, we arrive at
$$ \delta_T(\mathcal{G}(X),\mathcal{G}(Y)) \leq \gamma\delta_T(X,Y), $$
where $\gamma = \max\{\frac{\alpha_1}{\alpha_1+\mu^{-1}},\frac{\alpha_2}{\alpha_2+\mu\kappa(A)} \} < 1$.
Since the positive definite interval $\mathcal{I}$ is a compact set, we can choose $\gamma$ independently
from $X$ and $Y$. Specifically, since $\alpha_1\leq||\frac12(I+A^{-1})||_2$ and $\alpha_2\leq||\frac12(I+A)||_2$ then

$$\gamma=\max\left\{ \frac{1+||A^{-1}||_2}{1+||A^{-1}||_2+\mu^{-1}}, \frac{1+||A||_2}{1+||A||_2+\mu\kappa(A)} \right\} < 1,$$
which is strictly less than one for definite positive matrix $A$.
Thus, the map $\mathcal{G}$ is a strict contraction. Hence, from Banach contraction theorem it follows
that $\delta_T(X_k,X^{*})$ converges at linear rate given by $\gamma$, and $X_k\rightarrow X^*$.
\end{proof}

\begin{remark}
  If we choose the parameter as $\mu=\sqrt{\frac{1+||A||_2}{(1+||A^{-1}||_2)\kappa(A)}}$ to balance the arguments of the maximum defining $\gamma$ then we construct a theoretically optimal  convergence rate.
\end{remark}

\section{Numerical experiments}
\label{sec:4}
In this section, we report some numerical results associated to the two variants of our Algorithm \ref{alg1} FPM1 and FPM2 (the iterative schemes \eqref{eq3_sec2}--\eqref{eq5_sec2} respectively) and compare with some existing methods on the literature such as the Sra's iteration \eqref{Sra}, the scaled Newton method (given by \eqref{ScaledNewton} with $\mu_0 = 0.5$, and the gradient method (GradM) proposed in \cite{jain2015computing} with Barzilai--Borwein \cite{cruz2003nonmonotone} step--size corrected with Armijo line search, in order to demonstrate the effectiveness of our proposal on three different experiments. All test problems were performed on a intel(R) CORE(TM) i7--4770, CPU 3.40 GHz with 500GB HD and 16GB RAM, and all methods were implemented in Matlab.\\

In all experiments, presented in this section, in addition to checking the residual norm $E_k = ||A-X_k^{2}||_F$, we also compute the relative change of the two consecutive iterates
\begin{equation}
  reschg_k = \frac{||X_{k+1} - X_k||_F}{||X_k||_F}. \label{eq1_sec4}
\end{equation}

We let all algorithms run up to $N$ iterations and stop them at iteration $k < N$ if $E_k < \epsilon$, or $reschg_k < \epsilon_{X}$. We use the default values $N = 1000$, $\epsilon = 1$e-5 and $\epsilon_{X} = 1$e-6. Furthermore, for our procedures FPM1 and FPM2, we set $\mu = \nu\sqrt{ n + \kappa(A)}$  and $\mu = \nu\left(\frac{tr(A)}{n} + \frac{||A||_F}{n}\right)$ with $\nu\in(0,1)$ respectively, where $tr(A)$ denotes the trace of $A$. In addition, for all experiments and for all methods, we use the starting point $X_0 = (1/2)(A + I)$.

\subsection{Randomly generated symmetric positive definite problems}
In this subsection, we test the performance of all methods on problems of the form \eqref{problem} with $A\in\mathbb{R}^{n\times n}$ generated as follow, $A = QDQ^{\top}$ where
$$ Q = (I - 2w_1w_1^{\top})(I - 2w_2w_2^{\top})(I - 2w_3w_3^{\top}),$$
where $w_1$, $w_2$ and $w_3$ are three $n$-dimensional vector randomly generated in the unitary sphere, and $D\in\mathbb{R}^{n\times n}$  is a diagonal matrix $D = diag(\lambda_1,\lambda_2,\ldots,\lambda_n)$ whose $i$--th eigenvalue is defined by
$$ \log(\lambda_i) = \left(\frac{i-n}{n-1}\right)ncond. $$
The parameter \emph{ncond} in the above equality specifies the condition number of
$A$. Note that in such kind of problems the logarithms of the eigenvalues (and not the eigenvalues) are uniformly distributed, leading to problems which are typically harder to solve. In addition, observe that the optimal solution is $A^{1/2} = QD^{1/2}Q^{\top}$, due to $Q$ is an orthogonal matrix.\\

To illustrate the behaviour of the five methods, we show in Figure \ref{fig1} the residual norm $E_k$ along the iteration for a randomly generated problem with $n = 100$ and $ncond = 6$. In this Figure, we observe that the faster procedure is the Newton's method, which is expected due to its quadratic convergence behaviour. We also see that our FPM1 reduces the residual norm $E_k$ more quickly than the Sra's iteration. In addition, we note that the gradient method it produces a very slow decrease in the residual norm but the estimated solution is far from the $A^{1/2}$.\\

\begin{figure}
  \centering
  \includegraphics[width=8cm]{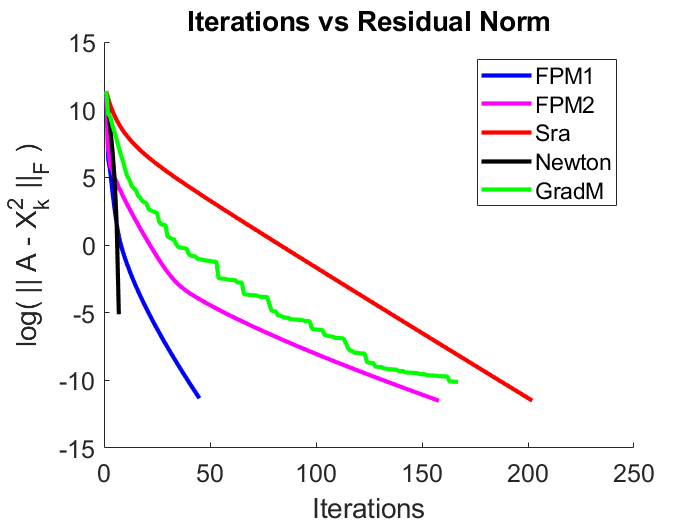}
  \caption{Behavior of the algorithms for $n=100$, $ncond = 6$. The y-axis is on a logarithmic scale.}
  \label{fig1}
\end{figure}

Table \ref{tab:1} contains the numerical results associated to this experiment for three different values of $n=100,500,1000$ and varying $ncond=1,3,5,10$. For each pair $(n,ncond)$, we generate ten independent problems and the we report the average number of iterations (Nitr), the average CPU time in seconds and the average of the residual norm (Error), that is, $Error = (1/10)\sum_{i=1}^{10} E(\hat{X}_i)$ where $E(\hat{X}_i) = ||A - \hat{X}_i^{2}||_F$ and $\hat{X}_i$ denotes the estimated solution obtained by the algorithm solving the $i$--th problem.\\

In order to compare the efficiency of the algorithms, we adopt the performance profile \cite{dolan2002benchmarking} introduced by Dolan and More to illustrate the whole performance of the all methods for the 120 problems tested in this subsection.

\begin{table}
\centering
\label{tab:1}
\caption{Numerical results for symmetric positive definite problems.}
\resizebox{15cm}{!}{\begin{tabular}{c|c c c c c| c c c c c}
  \hline
  	& \textbf{FPM1}	&	\textbf{FPM2}	&	\textbf{Sra}	&	\textbf{Newton}	&	\textbf{GradM}	&	\textbf{FPM1}	&	\textbf{FPM2}	&	\textbf{Sra}	&	\textbf{Newton}	&	\textbf{GradM} \\
  \hline
  & \multicolumn{5}{c|}{n = 100, ncond = 1} & \multicolumn{5}{|c}{n = 100, ncond = 3}\\
  \textbf{Nitr}	&22	      &	7	     &	20	     &	4	     &	28	     &	23	     &	25	     &	45	     &	6	     &	72	    \\
  \textbf{Time}	&0.01	  &	0.004	 &	0.02	 &	0.002	 &	0.014	 &	0.012	 &	0.013	 &	0.04	 &	0.003	 &	0.037	 \\
  \textbf{Error}&7.97e-6  &	7.44e-6  &	6.63e-6  &	4.82e-6  &	6.45e-6  &	6.99e-6  &	7.40e-6  &	7.94e-6  &	8.28e-8  &	8.28e-6  \\
  \hline
  & \multicolumn{5}{c|}{n = 100, ncond = 5} & \multicolumn{5}{|c}{n = 100, ncond = 10}\\
  \textbf{Nitr}	&32       &	115	     &	120	     &	8	     &	288	     &	292	     &	$>$2000	&	1716	 &	Fail	&	$>$2000	\\
  \textbf{Time}	&0.02	  &	0.05	 & 	0.1	     &	0.01	 &	0.16	 &	0.15	 &	0.95	&	1.48	 &	Fail	&	1.66	\\
  \textbf{Error}&8.70e-6  &	9.45e-6  &	9.03e-6  &	2.82e-10 &	9.83e-6  &	9.73e-6  &	0.0349	&	9.94e-6  &	Fail	&	0.1042	\\
  \hline
  & \multicolumn{5}{c|}{n = 500, ncond = 1} & \multicolumn{5}{|c}{n = 500, ncond = 3}\\
  \textbf{Nitr}	&53	      &	8	     &	21	     &	5	     &  	31	 &	51	     &	25	     &	47	     &	6	     &	52	\\
  \textbf{Time}	&0.89	  &	0.16	 &	0.61	 &	0.07	 &	0.61	 & 	0.85	 &	0.51	 &	1.36	 &	0.08	 &	1.07	\\
  \textbf{Error}&8.15e-06 &	7.84e-6  &	7.66e-6  &	1.49e-12 &	5.35e-6  &	9.09e-6  &	8.42e-6  &	7.90e-6  &	1.62e-7  &	8.35e-6	\\
  \hline
  & \multicolumn{5}{c|}{n = 500, ncond = 5} & \multicolumn{5}{|c}{n = 500, ncond = 10}\\
  \textbf{Nitr}	&56	      &	114	     &	124	     &	8	     &	324	     &	317	     &	$>$2000	 &	1747	 &	Fail	&	$>$2000	\\
  \textbf{Time}	&0.93	  &	2.29	 &	3.52	 &	0.11	 &	7.23	 &	5.56	 &	40.59	 &	51.31	 &	Fail	&	69.28	\\
  \textbf{Error}&8.61e-6  &	9.82e-6  &	9.18e-6  &	7.70-11  &	9.79e-6  &	9.96e-6  &	2.69e+4  &	9.95e-6  &	Fail	&	22.2072	\\
  \hline
  & \multicolumn{5}{c|}{n = 1000, ncond = 1} & \multicolumn{5}{|c}{n = 1000, ncond = 3}\\
  \textbf{Nitr}	&76	      &	9	     &	22	     &	5	     &	31	     &	74	     &	26	     &	48	     &	6	     &	78	\\
  \textbf{Time}	&6.82	  &	0.98	 &	3.51	 &	0.36	 &	3.48	 &	6.84	 &	2.98	 &	7.91	 &	0.45	 &	9.67	\\
  \textbf{Error}&9.84e-6  &	2.82e-6  &	5.68e-6  &	2.09e-12 &	7.56e-6  &	8.71e-6  &	6.80e-6  &	7.69e-6  &	2.26e-7  &	9.70e-6	\\
  \hline
  & \multicolumn{5}{c|}{n = 1000, ncond = 5} & \multicolumn{5}{|c}{n = 1000, ncond = 10}\\
  \textbf{Nitr}	&76	      &	116	     &	126	     &	8	     &	226	     &	331	     &	$>$2000	 &	1768	 &	Fail     &	$>$2000	\\
  \textbf{Time}	&6.95	  &	13.2	 &	20.73	 &	0.61 	 &	30.91	 &	29.94	 &	225.15   &	292.18   &	Fail     &	341.59	\\
  \textbf{Error}&9.91e-6  &	9.56e-6  &	9.31e-6  &	9.16e-11 &	9.47e-6  &	9.83e-6  &	4.51e+4  &	9.96e-6  &	Fail     &	70.8521	\\
\hline
\end{tabular}}
\end{table}

\begin{figure}
  \centering
  \begin{center}
  \subfigure[Performance profile based on the number of iterations]{\includegraphics[width=6cm]{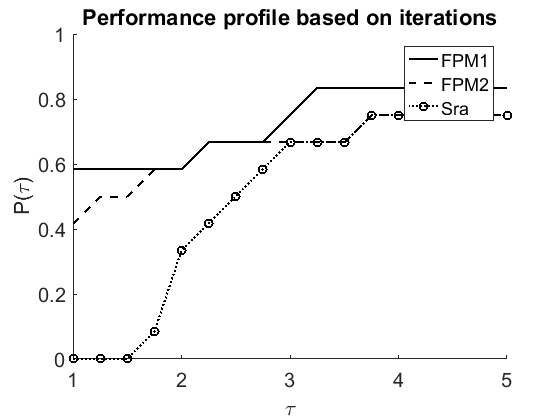}}  \hspace{1.5cm}
  \subfigure[Performance profile based on CPU--time]{\includegraphics[width=6cm]{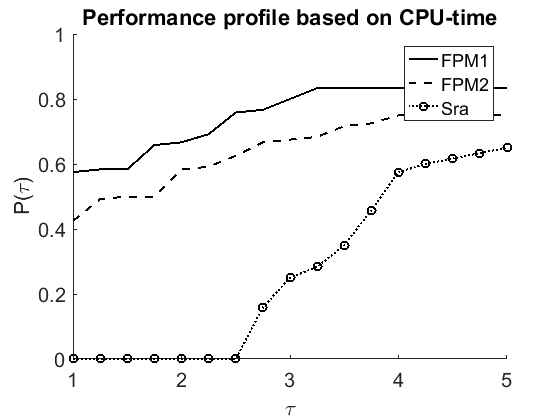}}
  \end{center}
  \caption{Performance profile based on the number of iterations and CPU--time, respectively.}
  \label{fig:1}
\end{figure}

\subsection{Random correlation and low--rank matrices}
In this subsection, we test the all methods on random generated positive definite matrices built with the following matlab command:

\begin{itemize}
  \item Example 1.1: random correlation matrices $A = \verb"gallery"('\verb"randcorr"',n))$.
  \item Example 1.2: $A = \verb"eye"(n) + \beta UU^{\top}$, where $U$ is a low--rank matrix, namely $U=\verb"randn"(n,k)$ with $k=10$, and a variable $\beta = \verb"rand"$.
  \item Example 1.3: the Hilbert matrix $A = \verb"hilb"(n)$,
\end{itemize}

For examples 1.1 and 1.2, we vary $n$ in $\{100,250,500,1000\}$, and compare the average number of iteration, the average CPU time in second and the average error $E(\hat{X}_i) = ||A - \hat{X}_i^2||_F$ obtained by the algorithms on a total of $30$ independent instances for each value of $n$. Note that the random correlation matrices are well--conditioned, the matrices given by  Example 1.2 are moderately well--conditioned while Hilbert's matrix is ill--conditioned. These test experiments were taken from \cite{Sra}.\\

Table \ref{tab:2} reports the numerical results associated to the test examples 1.1--1.2. From this table, we can see that the Newton's method obtains the best results both in CPU--time and in the number of iterations performed. Furthermore, we note that our FPM1 outperforms the other first--order approaches both in terms of iterations and CPU time. In fact, we observe that our FPM1 performs almost the same number of iterations as Newton's method for low--rank type problems. We also note that our second proposal converges very slowly for random correlation matrices, while for test examples 1.2 this procedure is faster than the Sra's iteration.\\

\begin{table}
\centering
\label{tab:2}
\caption{Numerical results for Examples 1.1 and 1.2.}
\resizebox{15cm}{!}{\begin{tabular}{c|c c c c c| c c c c c}
  \hline
  	& \textbf{FPM1}	&	\textbf{FPM2}	&	\textbf{Sra}	&	\textbf{Newton}	&	\textbf{GradM}	&	\textbf{FPM1}	&	\textbf{FPM2}	&	\textbf{Sra}	&	\textbf{Newton}	&	\textbf{GradM} \\
  \hline
  & \multicolumn{5}{c|}{Example 1.1, $n = 100$} & \multicolumn{5}{|c}{Example 1.1, $n = 250$}\\
\textbf{Nitr}	&	26	    & 	212	    &	49	    &	5	    &	100	    &	30	    &	210	    &	56	    &	5	    &	123	    \\
\textbf{Time}	&	0.01	&	0.11	&	0.05	&	0.002	&	0.04	&	0.09	&	0.73	&	0.32	&	0.02	&	0.34    \\
\textbf{Error}	&	8.01e-6	&	9.45e-6	&	9.13e-6	&	1.99e-6	&	7.69e-6	&	8.59e-6	&	9.57e-6	&	9.20e-6	&	2.02e-6	&	7.41e-6	\\
  \hline
  & \multicolumn{5}{c|}{Example 1.1, $n = 500$} & \multicolumn{5}{|c}{Example 1.1, $n = 1000$}\\
\textbf{Nitr}	&	50	    &	467	    &	82	    &	6	    &	203	    &	56 	    &	563	    &	91	    &	6	    &	232	    \\
\textbf{Time}	&	0.84	&	9.37	&	2.30	&	0.08	&	3.18	&	5.15	&	61.54	&	14.74	&	0.47	&	22.08	\\
\textbf{Error}	&	9.24e-6	&	9.86e-6	&	9.46e-6	&	2.17e-6	&	8.21e-6	&	9.39e-6	&	9.92e-6	&	9.64e-6	&	1.54e-6	&	7.47e-6	\\
  \hline
  & \multicolumn{5}{c|}{Example 1.2, $n = 100$} & \multicolumn{5}{|c}{Example 1.2, $n = 250$}\\
\textbf{Nitr}	&	14	    &	26	    &	89	    &	6	    &	19	    &	12	    &	58	    &	135	    &	7	    &	18	\\
\textbf{Time}	&	0.01	&	0.01	&	0.08	&	0.003	&	0.01	&	0.04	&	0.19	&	0.76	&	0.02	&	0.05	\\
\textbf{Error}	&	5.30e-6	&	7.31e-6	&	9.25e-6	&	3.07e-7	&	4.64e-6	&	5.47e-6	&	8.58e-6	&	9.33e-6	&	7.72e-7	&	4.66e-6	\\
  \hline
  & \multicolumn{5}{c|}{Example 1.2, $n = 500$} & \multicolumn{5}{|c}{Example 1.2, $n = 1000$}\\
\textbf{Nitr}	&	13	    &	111	    &	178	    &	7	    &	18	    &	12	    &	698	    &	230	    &	7	    &	17	\\
\textbf{Time}	&	0.21	&	2.26	&	5.00	&	0.10	&	0.29	&	1.06	&	36.59	&	37.02	&	0.57	&	1.65	\\
\textbf{Error}	&	4.95e-6	&	9.34e-6	&	9.36e-6	&	5.48e-7	&	5.35e-6	&	4.50e-6	&	6.21e-6	&	9.54e-6	&	9.28e-7	&	3.51e-6	\\
\hline
\end{tabular}}
\end{table}

\begin{figure}
  \centering
  \begin{center}
  \subfigure[Residual error for random correlation matrices with $n=100$]{\includegraphics[width=6cm]{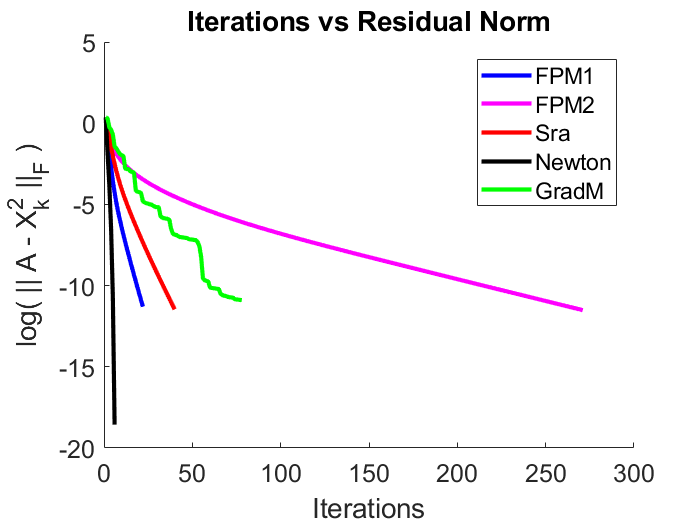}}  \hspace{1.5cm}
  \subfigure[Residual error for low--rank matrices with $n=100$]{\includegraphics[width=6cm]{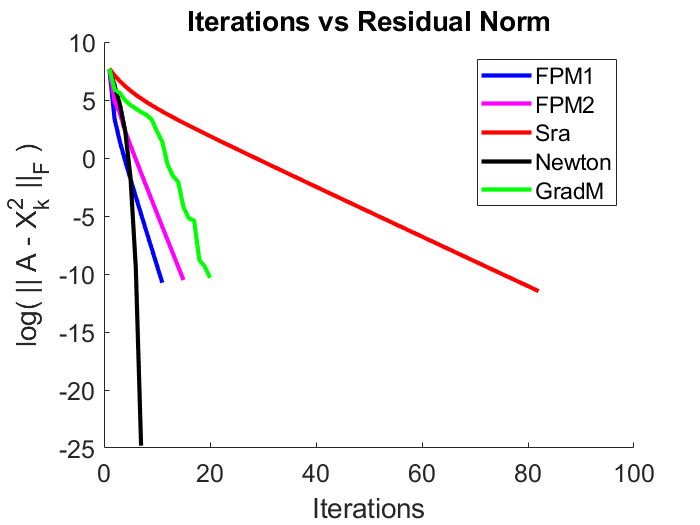}} \hspace{1.5cm}
  \subfigure[Residual error for the Hilbert matrix with $n=50$]{\includegraphics[width=6cm]{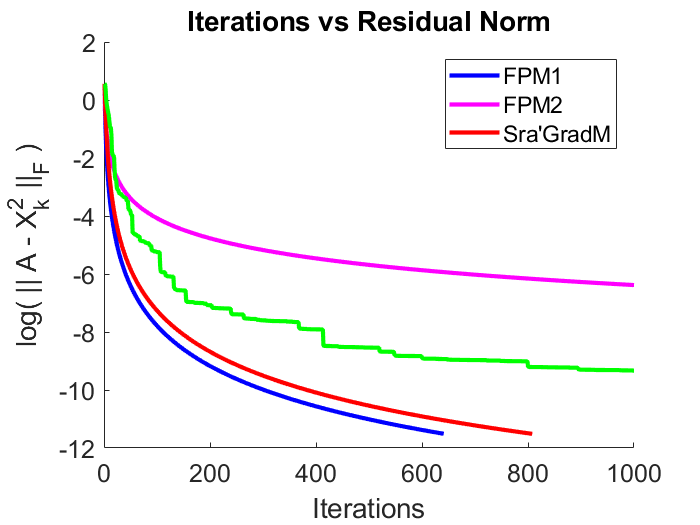}}
  \end{center}
  \caption{Residual error vs the iterations number for all methods. The y-axis is on a logarithmic scale.}
  \label{fig:3}
\end{figure}

The residual norm of the iterates for the three different examples 1.1, 1.2 and 1.3
are shown in Figure \ref{fig:3}. In the subfigure (c) we omit the curve associated with Newton's method because this procedure fails in ill--conditioned problems. Figure \ref{fig:3} shows that our first proposal FPM1 converges faster than the rest of the first--order methods, while the Newton's method is superior to the rest of methods when the matrix $A$ is well--conditioned or moderately well--conditioned. In addition, we observe that in the ill--conditioned situation, our FPM1 method has a very similar behavior to the Sra's iteration, however our proposal makes a smaller number of matrix inversions and therefore is the most efficient method in this case.

\section{Conclusion}
\label{sec:5}
The goal of this paper is to develop an efficient algorithm which is
able to compute the square root of a given symmetric semi--definite positive matrix. Our
strategy is simply to rearrange the nonlinear equation $A - X^2 = 0$ in order to design a contractive mapping which is regulated for an exogenous (conveniently chosen) parameter, leading to effective scaled fixed point methods, which only require calculating one matrix inverse (numerically, solving a linear system of equations) and one matrix product per iteration. We further theoretically show the global convergence of one of our proposed numerical methods. In fact, we demonstrate that this first proposal converges Q--linearly to $A^{1/2}$. The second one is only numerically studied. However, our numerical experiments on randomly generated symmetric positive definite matrices, with different conditioning situations, show that our two procedures are effective and efficient to solve the matrix equation $A - X^2 = 0$. In addition, our numerical tests show that our first proposal (FPM1) outperforms several first--order methods existing in the literature.

\begin{acknowledgements}
This research was supported in part by Conacyt, Mexico (258033 research grant and H.O.L. PhD.
studies scholarship). The second author wants to thank the Federal University of Santa Catarina--Brazil and remarks that his contribution to the present article was predominantly carried out at this institution.
\end{acknowledgements}


\begin{thebibliography}{}
\bibitem{Ando} Ando T.: Fixed points of certain maps on positive semidefinite operators. Functional Analysis and Approximation, Springer, 29--38 (1981).
\bibitem{deadman2012blocked} Deadman E., Higham N. J. and Ralha R.: Blocked Schur algorithms for computing the matrix square root. International Workshop on
Applied Parallel Computing. \textbf{1}, 171--182 (2012).
\bibitem{dolan2002benchmarking} Dolan E. D. and Mor{\'e} J. J.: Benchmarking optimization software with performance profiles. Mathematical programming.
\textbf{91}, 201--213 (2002).
\bibitem{gawlik2019zolotarev} Gawlik E. S.: Zolotarev Iterations for the Matrix Square Root. SIAM journal on matrix analysis and applications. \textbf{40}, 696--719 (2019).
\bibitem{higham1986newton} Higham N. J.: Newton’s method for the matrix square root. Mathematics of Computation. \textbf{46}, 537--549 (1986).
\bibitem{higham1997stable} Higham N. J.: Stable iterations for the matrix square root. Numerical algorithms. \textbf{15}, 227--242 (1997).
\bibitem{Higham} Higham N. J.: Functions of matrices: theory and computation. SIAM (2008).
\bibitem{iannazzo2017riemannian} Iannazzo B. and Porcelli M.: The Riemannian Barzilai--Borwein method with nonmonotone line search and the matrix geometric mean
computation. IMA Journal of Numerical Analysis. \textbf{38}, 495--517 (2017).
\bibitem{cruz2003nonmonotone} La Cruz W. and Raydan M.: Nonmonotone spectral methods for large--scale nonlinear systems. Optimization Methods and Software.
\textbf{18}, 583--599 (2003).
\bibitem{laszkiewicz2009pade} Laszkiewicz B. and Zietak K.: A Padé family of iterations for the matrix sector function and the matrix pth root. Numerical Linear
Algebra with Applications. \textbf{16}, 951--970 (2009).
\bibitem{lee2008invariant} Lee H. and Lim Y.: Invariant metrics, contractions and nonlinear matrix equations. Nonlinearity, \textbf{21}, 857--878 (2008).
\bibitem{lemmens2012nonlinear} Lemmens B. and Nussbaum R.: Nonlinear Perron--Frobenius Theory. Cambridge university press. \textbf{189}, (2012).
\bibitem{lim2012matrix} Lim Y. and P{\'a}lfia M.: Matrix power means and the Karcher mean. Journal of Functional Analysis. \textbf{262}, 1498--1514 (2012).
\bibitem{PanChen} Pan V. Y., Chen Z., Zheng A. and others.: The complexity of the algebraic eigenproblem. Mathematical Sciences Research Institute Berkeley. 1998--71
(1998).
\bibitem{jain2015computing} Prateek J., Chi J., Sham K. and Praneeth N.: Global Convergence of Non--Convex Gradient Descent for Computing Matrix Squareroot.
Proceedings of the 20th International Conference on Artificial Intelligence and Statistics. \textbf{54}, 479--488 (2017).
\bibitem{sra2015conic} Sra S. and Hosseini R.: Conic geometric optimization on the manifold of positive definite matrices. SIAM Journal on Optimization.  \textbf{25},
713--739 (2015).
\bibitem{Sra} Sra S.: On the matrix square root via geometric optimization. The electronic journal of linear algebra. \textbf{3}, 433--443 (2016).
\bibitem{sra2016positive} Sra S.: Positive definite matrices and the S--divergence. Proceedings of the American Mathematical Society. \textbf{144}, 2787--2797 (2016).
\bibitem{VanDerMerwe} Van Der Merwe R. and Wan E. A.: The square--root unscented Kalman filter for state and parameter--estimation. IEEE international
conference on acoustics, speech, and signal processing. Proceedings (Cat. No. 01CH37221), \textbf{6}, 3461--3464 (2001).
\bibitem{yuan2016riemannian} Yuan X., Huang W., Absil P.--A. and Gallivan K. A.: A Riemannian limited-memory BFGS algorithm for computing the matrix geometric mean.
Procedia Computer Science. \textbf{80}, 2147--2157 (2016).
\bibitem{zhu2017riemannian} Zhu X.: A Riemannian conjugate gradient method for optimization on the Stiefel manifold. Computational Optimization and Applications.
\textbf{67}, 73--110 (2017).
\end{thebibliography}
\end{document}